\documentclass[a4paper,12pt]{article}

\usepackage[utf8]{inputenc}
\usepackage[T1]{fontenc}
\usepackage{amsmath,amssymb,amsfonts,amsthm,mathrsfs}
\usepackage{geometry}
\usepackage[english]{babel}
\usepackage{tikz}
\usetikzlibrary{decorations.markings}

\newenvironment{resume}[1]{
	\begin{list}{}{
		\setlength{\leftmargin}{1cm}
		\setlength{\rightmargin}{1cm}
	}\item[]
	{\bf #1}
	}{\end{list}}

\theoremstyle{plain}
\newtheorem{df}{Definition}
\newtheorem{thm}{Theorem}
\newtheorem{lem}{Lemma}
\newtheorem{prop}{Proposition}

\newtheorem{rmq}{Remark}
\newtheorem{cor}{Corollary}
\newtheorem*{thm*}{Theorem}

\newcommand\mb{\mathbb}
\newcommand\mc{\mathcal}
\newcommand\mf{\mathfrak}
\newcommand\mr{\mathrm}
\newcommand\ms{\mathscr}

\newcommand\F{\mc{F}}
\newcommand\G{\mc{G}}
\newcommand\C[1]{\mb{C}^{#1}}
\newcommand\OC[1]{\mc{O}_{(\C{#1},0)}}
\newcommand\mC[1]{\mf{m}_{(\C{#1},0)}}
\newcommand\Df[1]{\mr{Diff}(\C{#1},0)}
\newcommand\Cg[1]{(\C{#1},0)}

\title{
Pairs of Morse functions
}

\author{Olivier Thom}
\date{}

\begin{document}
\maketitle

\begin{resume}{Abstract.}
The goal of this paper is to classify pairs of Morse functions in general position modulo the action of different groups.
In particular, we obtain the classification of generic pairs of Morse functions, with or without target diffeomorphisms, and that of quotients of Morse functions.

We will also present a lemma which gives a sufficient condition for two pairs of functions to be conjugated.
\end{resume}

\section{Introduction}

Throughout this paper, we will denote by $\OC{n}$ the set of germs at $0$ of holomorphic functions on $\C{n}$ and by $\mC{n}$ its maximal ideal.
We will also use the notation $X\cdot f = d_X f$ to mean the derivative of $f$ in the direction given by the holomorphic vector field $X$.

For a group $\ms{S}$ acting on pairs of functions $f,g\in \OC{n}$, we will say that two pairs $p_1$ and $p_2$ are $\ms{S}$-conjugated or $\ms{S}$-equivalent if there exists $\varphi\in\ms{S}$ such that $\varphi\cdot p_1=p_2$.
In this paper we will consider the groups $\ms{R}, \ms{A}, \ms{F}, \ms{Q}$ which follows:
$\ms{R}=\Df{n}$ acting by composition at the source, $\ms{A}=\Df{n}\times\Df{2}$ acting by composition at the source and at the target, $\ms{F}=\Df{n}\times(\Df{})^2$ acting by $(\varphi,\psi_1,\psi_2)\cdot (f,g) = (\psi_1\circ f,\psi_2\circ g)\circ \varphi^{-1}$, and $\ms{Q} = \Df{n}\rtimes\OC{n}^*$ acting by $(\varphi,U)\cdot p=Up\circ \varphi^{-1}$.
Classification of pairs of functions up to $\ms{F}$-equivalence corresponds to the classification of pairs of foliations up to diffeomorphism at the source; classification of pairs of functions $(f,g)$ up to $\ms{Q}$-equivalence corresponds to the classification of meromorphic functions $f/g$ up to diffeomorphism at the source.

Let $f$ and $g$ be two Morse functions on $(\C{n},0)$ whith quadratic parts $q_f$ and $q_g$.
Denote by $\F$ and $\G$ the foliations given by the level sets of $f$ and $g$.
Denote also by $I(f,g)$ the tangency ideal between $f$ and $g$, that is the ideal of $\OC{n}$ spanned by $(\partial_{x_i}f \partial_{x_j}g- \partial_{x_j}f \partial_{x_i}g)_{i,j}$ for a set of coordinates $(x_i)$ and by $\mr{Tang}(f,g) = \mr{Tang}(\F,\G)$ the set of zeroes of $I(f,g)$, which we will name the tangency locus between $f$ and $g$. 
\\

We will begin by giving the classification up to $\ms{R}$-equivalence of pairs of Morse functions, but first, let us recall the well-known classification of pairs of quadratic forms on $\C{n}$ (cf \cite{hp}).
Seen as matrices, two nondegenerate forms $q_f$ and $q_g$ can be simultaneously diagonalized by blocks with blocks
\[ \left(\begin{aligned}
&(0) & & & & & 1\\
& & & &\text{\reflectbox{$\ddots$}} & & \\ 
& &\text{\reflectbox{$\ddots$}} & & & & \\
&1 & & & & &(0)\\
\end{aligned} \right)
\quad\text{ and }\quad
\left(\begin{aligned}
&(0) & & & & & \lambda\\
& & & &\text{\reflectbox{$\ddots$}} & & 1\\
& &\text{\reflectbox{$\ddots$}} & &\text{\reflectbox{$\ddots$}} & & \\
&\lambda &1 & & & &(0)\\
\end{aligned} \right).
\]

As an example, take the quadratic forms given by the matrices
\[\left(\begin{aligned} &0 &1\\ &1 &0\\ \end{aligned}\right)\quad \text{and}\quad \left(\begin{aligned} &0 &1\\ &1 &1\\ \end{aligned}\right):\]
$f=2xy$ and $g=2xy+y^2$.
We see that this pair cannot be simultaneously diagonalized.

Nevertheless, counting the parameters in the diagonalization by blocks we see that a generic (outside a set of codimension $1$) pair of quadratic forms $(q_f,q_g)$ can be simultaneously diagonalized.

Morse theorem (\cite{morse}) allows us to assume without loss of generality that $f=\sum{x_i^2}$.
Moreover we suppose that $q_f$ and $q_g$ are in generic position: $q_f(x) = \sum{x_i^2}$, $q_g(x) = \sum{\lambda_i x_i^2}$ with $\lambda_i\neq \lambda_j\neq 0$ if $i\neq j$ up to a linear change of coordinate.

Next, look at the tangency locus between the foliations $\F$ and $\G$: if $f$ and $g$ were diagonal quadratic forms, this would be the reunion of the coordinate axes.
In general, if $q_f$ and $q_g$ are diagonal, it is diffeomorphic and tangent to the reunion of the axes so we can suppose that it is exactly the reunion of the axes; this will be detailed further.

For example, in the case $n=2$ the functions $f=x^2+y^2$ and $g=x^2 + 2y^2$ give the following real phase portrait:

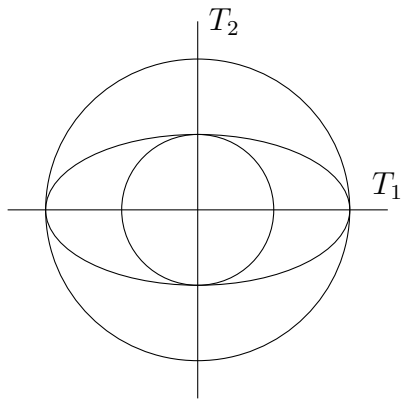
\begin{center}
\begin{tikzpicture}
\draw (0,0) circle (1);
\draw (0,0) circle (2);
\draw (0,0) ellipse (2 and 1);
\draw (0,2.5) -- (0,-2.5);
\draw (-2.5,0) -- (2.5,0);
\draw (2.5,0) node[above] {$T_1$};
\draw (0,2.5) node[right] {$T_2$};
\end{tikzpicture}
\end{center}

If we name the axes $T_j$ as in the picture, we can look at the restriction of each function to each tangency curve, which gives couples $(f\vert_{T_j},g\vert_{T_j})$ for each $j$.
If $\Phi$ is a diffeomorphism of $(\C{n},0)$ stabilizing the $T_j$'s, we have $((f\circ \Phi)\vert_{T_j},(g\circ \Phi)\vert_{T_j})= (f\vert_{T_j},g\vert_{T_j})\circ (\Phi\vert_{T_j})$ so that each couple $(f\vert_{T_j},g\vert_{T_j})$ up to diffeomorphism on the right gives an invariant for the $\ms{R}$-equivalence of pairs of functions.

Hence, if $C_0$ and $C_1$ are smooth curves, we will say that two couples $(u_0,v_0)$ and $(u_1,v_1)$ with $u_j,v_j \in \mc{O}(C_j,0)$ are conjugated under the action of $\mr{Diff}(C_0,C_1)$ on the right if there exists $\psi\in \mr{Diff}(C_0,C_1)$ such that $(u_0,v_0) = (u_1,v_1)\circ \psi$.

These invariants are enough to classify the pairs of Morse functions up to $\ms{R}$-equivalence, as stated in the theorem :

\begin{thm*}
Let $(f_0,g_0)$ and $(f_1,g_1)$ be two pairs of Morse functions on $(\C{n},0)$ with quadratic parts $(q_{f_i},q_{g_i})$ in generic position.
Suppose that we can number the tangency curves $T_j^i$ ($j=1,\ldots,n$ and $i=0,1$) in such a manner that the pairs of Morse functions $(f_i\vert_{T_j^i},g_i\vert_{T_j^i})$ are conjugated under the action of $\mr{Diff}(T_j^0,T_j^1)$ on the right.
Then $(f_0,g_0)$ and $(f_1,g_1)$ are $\ms{R}$-equivalent.
\end{thm*}

As a consequence, if two pairs of Morse functions with quadratic parts in generic position are topologically conjugated, they are analytically conjugated.
\\

A part of the proof of this theorem is in fact quite general and is expressed as a separate lemma (the key lemma in what follows); the section \ref{sec_key_lemma} is devoted to the statement and proof of this lemma.
The next section (section \ref{sec_functions}) handles the $\ms{R}$-classification of pairs of Morse functions.

After the $\ms{R}$-classification of pairs of Morse functions, the $\ms{A}$-classification and the $\ms{F}$-classification are just a matter of rewriting as it will be shown later; these are done in sections \ref{sec_foliations} and \ref{sec_A_classification}.
The $\ms{Q}$-classification of pairs of Morse functions is not a straightforward consequence of the former theorem; the main result is that a generic pair $(f,g)$ of Morse functions is determined up to the action of $\ms{Q}$ by the $3$-jets of $f$ and $g$, so that a generic quotient of Morse functions is diffeomorphic to an explicit rational function of degree $3$.
This will be detailed in section \ref{sec_quotients}.

We will also show that the restriction of a generic Morse function to a quadratic cone (the set of zeroes of a Morse function) is determined up to diffeomorphism by its quadratic part (in section \ref{sec_restrictions}).

In the last section, we will show that the key lemma can be applied in a general setting, by rediscovering classical results like the classification of folds, or giving finite determinacy results.
As an example, we will give the classification of some special pairs of cusps.

Some of these problems can be restated in terms of diagrams in the sense of Dufour (cf. \cite{dufour}): the $\ms{F}$-classification of pairs of Morse functions corresponds to the classification of divergent diagrams of Morse functions
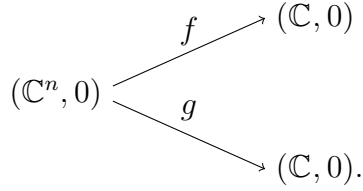
\begin{center}
\begin{tikzpicture}
\draw (0,0) node[left] {$\Cg{n}$};
\draw (2,1) node[right] {$\Cg{}$};
\draw (2,-1) node[right] {$\Cg{}.$};
\draw[->] (0,0.1) -- (2,1);
\draw[->] (0,-0.1) -- (2,-1);
\draw (1,0.5) node[above] {$f$};
\draw (1,-0.5) node[above] {$g$};
\end{tikzpicture}
\end{center}

We should also mention the work of J. Vey about a similar problem: the simultaneous reduction of a Morse function and a volume form (cf. \cite{vey}).

\section{Proof of the key lemma}
\label{sec_key_lemma}

In this section we want to prove the following:

\begin{lem}[Key Lemma]
\label{key_lemma}
Let $f$, $g_0$ and $g_1$ be three functions on $(\C{n},0)$ where $f$ has a singular point at $0$.
Suppose that the tangency ideals $I(f,g_0)$ and $I(f,g_1)$ are equal and that $g_1-g_0\in I(f,g_0)$.
Then $(f,g_0)$ and $(f,g_1)$ are $\ms{R}$-conjugated.
\end{lem}

The proof of this lemma is based on Moser's path method: we will construct a path $(f,g_t)$ between $(f,g_0)$ and $(f,g_1)$ and show that every $(f,g_t)$ are diffeomorphic.
Put $g_t=g_0 + t(g_1-g_0)$ and $g(t,\cdot) = g_t(\cdot)\in\mc{O}(U)$ for a neighborhood $U$ of $[0,1]\times\{0\}$ in $\mb{C}_t\times\mb{C}^n$.
Introduce also $I = I(f,g)$ (which is an ideal of $\mc{O}(U)$) and for each $t$, $I_t = I(f,g_t)$ (which is an ideal of $\OC{n}$).
Write finally $d_xf\wedge d_xg = \sum_{i<j}{h_{ij}dx_i\wedge dx_j}$ for a system of coordinates $(x_i)$ on $\C{n}$, $J = \langle h_{ij}\rangle_{i<j}$ and note that $I_t = \langle h_{ij}(t,\cdot)\rangle_{i<j}$.

We will first study these ideals to show that $J=I_0\otimes_{\mc{O}_x} \mc{O}(U)$ where $\mc{O}_x$ denotes the set of germs of holomorphic functions in the variables $x_1,\ldots,x_n$.

\begin{prop}
Suppose $I_0=I_1$, then $I_0=I_t$ for $t$ generic.
\end{prop}

\begin{proof}
The tangency ideal $I_t$ is spanned by the components of $df\wedge dg_t = t df\wedge dg_1 + (1-t) df\wedge dg_0$ so it is contained in $I_0$.
But $I_0/I_t$ is null for $t=0$ so the support of $I_0/I_t$ can only consist of finitely many points, hence the result.
\end{proof}

In what follows, we will use the additional hypothesis that $I_t$ is constant along the interval $[0,1]$.
If this is not the case, we could find a point $t_0\in \C{}$ such that $I_t=I_0$ for each $t$ in both segments $[0,t_0]$ and $[t_0,1]$ (thanks to the previous proposition) and use what will follow on these segments to show that $(f,g_0)\simeq (f,g_{t_0})\simeq (f,g_1)$ so we can indeed suppose without loss of generality that $I_t$ is constant along $[0,1]$.

\begin{prop}
For each $t_0$, the localization $J_{(t_0)}$ of $J$ at $t_0$ satisfies $J_{(t_0)}=I_{0}\otimes_{\mc{O}_x}\C{}\{t-t_0,x\}$.
\end{prop}

\begin{proof}
It is enough to prove that $J_{(t_0)} = I_{t_0}\otimes\C{}\{t-t_0,x\}$ because $I_{t_0}=I_0$.

Note first that the $h_{ij}$ are affine in $t$ so that $h_{ij}(t) = h_{ij}(t_0) + \frac{t-t_0}{1-t_0}(h_{ij}(1)-h_{ij}(t_0))$ (we supposed that $t_0\neq 1$, the case $t_0=1$ can be done similarly).
Denote by $H(t)$ the vector $(h_{ij}(t))_{i<j}$; the hypothesis that $I_1=I_{t_0}$ then gives a matrix $A$ with constant coefficients such that $H(1)=AH(t_0)$.
Hence the existence of a matrix $B$ satisfying $H(t) = (id + (t-t_0)B)H(t_0)$.

For $t$ near $t_0$, the matrix $id+(t-t_0)B$ is invertible so the components of the vectors $H(t)$ and $H(t_0)$ span the same germ of ideal around the point $t_0$.
Note finally that the germ of ideal spanned by the components of $H(t_0)$ is $I_0\otimes\C{}\{t-t_0,x\}$.
\end{proof}

As a corollary, for each point $p_0 = (t_0,x_0)\in U\subset\C{}_t\times\C{n}$, we have the relation $J_{(p_0)} = (I_0)_{(x_0)}\otimes_{\C{}\{x-x_0\}}\C{}\{t-t_0,x-x_0\}$.

\begin{prop}
\label{prop_coherence}
$J=I_0\otimes_{\mc{O}_x}\mc{O}(U)$.
\end{prop}

\begin{proof}
We can suppose that the neighborhood $U$ is Stein.
The ideal $J$ (resp. $I_0\otimes\mc{O}(U)$) defines a sheaf of ideals $\ms{J}$ (resp. $\ms{K}$) defined by $\ms{J}_{(p_0)} = J_{(p_0)}$ for $p_0\in U$ (resp. $\ms{K}_{(p_0)} = (I_0)_{(x_0)}\otimes\C{}\{t-t_0,x-x_0\}$ for $p_0=(t_0,x_0)\in U$).
These sheaves are locally of finite type; if $a_1,\ldots,a_k$ are local sections of $\ms{J}$ (resp. $\ms{K}$), the sheaf of relations $\mc{R}(a_1,\ldots,a_k)$ may be viewed as the relations of the sections $a_i$ of the sheaf $\mc{O}$.
Hence by Oka's theorem (see for example \cite{hormander}), $\mc{R}(a_1,\ldots,a_k)$ is locally of finite type and $\ms{J}$ and $\ms{K}$ are coherent.

Take $a\in I_0\otimes\mc{O}(U)$, then $a_{(p)}\in \ms{K}_{(p)} = \ms{J}_{(p)}$ for each $p\in U$; since $U$ is Stein and since the global sections $h_{ij}$ span $\ms{J}$ locally, there exists holomorphic $r_{ij}\in \mc{O}(U)$ such that $a=\sum{r_{ij}h_{ij}}$, ie. $a\in J$ (cf. \cite{hormander}).

The converse works in the same way with $h_{ij}(0,\cdot)$ as global sections spanning $\ms{K}$ locally.
\end{proof}

Moreover, if $g_1-g_0\in I_0$ as in the hypotheses of the lemma, $g_1-g_0\in J$ by the former proposition, so $J$ is also equal to $I$ because $df\wedge dg = d_xf\wedge d_xg + (g_1-g_0)df\wedge dt$.

Now we can prove the key lemma:

\begin{proof}[Proof of the key lemma]
As noted above, the hypothesis $g_1 -g_0\in I_0$ together with proposition \ref{prop_coherence} means that there exists holomorphic $r_{ij}(t,x)$ (for $i<j$) such that $g_1 - g_0 = \sum_{i<j}{r_{ij}h_{ij}}$.

To use the path method, we need to find a vector field $X = \sum_{i=1}^n X_i\partial_{x_i} + \partial_t$ defined in a neighborhood of $\{0\}\times [0,1]\subset \C{n}\times [0,1]$ such that $X\cdot f = X\cdot g = 0$.
We also want to have $X(0,t)=\partial_{t}$ so that the flow $\varphi_s(x,t)$ of $X$ will be defined on a neighborhood of $\{0\}\times[0,1]$.
The diffeomorphism $\varphi: x \mapsto \varphi_1(x,0)$ will then verify $(f\circ \varphi,g_0 \circ \varphi)=(f,g_1)$ on $\Cg{n}$.

Remember that 
\[ X\cdot f = \sum_{i=1}^n{X_i\partial_{x_i}f}\quad\text{and}\]
\[ X\cdot g = \sum_{i=1}^n{X_i\partial_{x_i}g_t} + (g_1-g_0).\]

Note that it is enough to find for each $j=2,\ldots,n$ a vector field $X^j$ satisfying $X^j\cdot f =0$ and
\[\sum_{i=1}^{n}{X^j_i \partial_{x_i}g_t} + \sum_{i=1}^{j-1}{r_{ij}h_{ij}} = 0\]
because the vector field $X = \sum_{j=2}^{n}{X^j} + \partial_{t}$ would then be as sought.

On $U_j = \{\partial_{x_j}f\neq 0\}$, we may impose
\[ X^j_j = \frac{-1}{\partial_{x_j}f}\left(\sum_{i\neq j}{(\partial_{x_i}f)X^j_i}\right) \]
so that
\begin{align*}
\left(\partial_{x_j}f\right) \left( \sum_{i=1}^{n}{X^j_i \partial_{x_i}g_t} + \sum_{i=1}^{j-1}{r_{ij}h_{ij}} \right) &= \sum_{i\neq j}{ \left(\partial_{x_j}f \partial_{x_i}g_t - \partial_{x_i}f \partial_{x_j}g_t \right)X^j_i} + (\partial{x_j}f) \left(\sum_{i=1}^{j-1}{r_{ij}h_{ij}} \right)\\
     &= \sum_{i\neq j}{-h_{ij}X^j_i} + (\partial_{x_j}f) \left(\sum_{i=1}^{j-1}{r_{ij}h_{ij}} \right).
\end{align*}

So we can choose $X^j_i = r_{ij} \partial_{x_j}f$ if $i<j$ and $X_i^j = 0$ for $i>j$ which gives $X^j_j = -\sum_{i< j}{r_{ij} \partial_{x_i}f}$.
We see that every component $X_i^j$ is holomorphic around $\{\partial_{x_j}f = 0\}$ which means that the vector field $X^j$ is defined on $(\C{n},0)\times[0,1]$.
Moreover, since $f$ is singular at $0$, every $\partial_{x_i}f$ cancels at $0$ so that each $X^j$ cancels on $\{0\}\times[0,1]$.

The vector field $X = \sum_j{X^j} + \partial_{t}$ is the one we wanted.
\end{proof}

\begin{rmq}
\label{rmq_lemma}
The hypothesis "$f$ has a singular point at $0$" is only used to show that the vector field $X - \partial_{t}$ cancels along the $t$-axis, which is also true if all the $r_{ij}$ cancel on $\{0\}\times[0,1]$.
It is also the case if $g_1-g_0$ cancels at a high enough order at the origin (the exact order depends on the coefficients $h_{ij}$).
\end{rmq}

\section{$\ms{R}$-classification of pairs of Morse functions}
\label{sec_functions}

A pair of Morse functions $(f,g)$ is called $\ms{R}$-generic if (up to linear isomorphism) the quadratic parts $q_f$ and $q_g$ are diagonal : $q_f(x) = \sum{x_i^2}$ and $q_g(x) = \sum{\lambda_ix_i^2}$ with $\lambda_i\neq \lambda_j$ if $i\neq j$.

Let $(f,g)$ be an $\ms{R}$-generic pair of Morse functions.
Let us first study the tangency loci: if $q_f$ and $q_g$ are diagonal, $\mr{Tang}(q_f,q_g)$ is the union of the coordinate axes; in general, we have the following:

\begin{prop}
\label{prop_locus1}
The sets $\mr{Tang}(f,g)$ and $\mr{Tang}(q_f,q_g)$ are diffeomorphic and tangent.
\end{prop}

\begin{proof}
We can suppose $f$ quadratic and $q_g$ diagonal.
Blow up the origin to get that (recycling the coordinates $x_i$ as coordinates in the blow-up) the transforms of $f$ and $g$ are given by
\[\tilde{f} = x_1^2(1+x_2^2+\ldots +x_n^2) \quad\text{ and }\quad \tilde{g} = x_1^2(\lambda_1 + \lambda_2x_2^2+\ldots) + x_1^3(\ldots).\]
We will simultaneously compute the tangency locuses $\mr{Tang}(f,g)$ and $\mr{Tang}(q_f,q_g)$ in the blow-up to show this proposition (since we already know $\mr{Tang}(q_f,q_g)$, this will help understand $\mr{Tang}(f,g)$).
Write $\hat{f}=\tilde{f}=\tilde{q_f}$ and $\hat{g}=\tilde{q_g}+x_1^3\varepsilon$ with $\varepsilon=0$ or $\varepsilon=x_1^{-3}(\tilde{g}-\tilde{q_g})$.

Note that the genericity hypothesis on the $n$-uple $(\lambda_1,\ldots,\lambda_n)$ implies that $q_f$ and $q_g$ are not tangent near a point of the surface $\{f=0\}$ (exept at $0$).
In the blow-up, put $S:=\{1+x_2^2+\ldots+x_n^2=0\}$ and $E:=\{x_1=0\}$.
The remark above tells that the components of the tangency locus between $\tilde{q_f}$ and $\tilde{q_g}$ which are different from $E$ do not intersect $E\cap S$.
So this is the case for $\hat{f}$ and $\hat{g}$ independently of $\varepsilon$.
The change of coordinate $x_1\mapsto \sqrt{\hat{f}}$ is allowed near each point of $E$ far away from the hypersurface $S$ and every component of $\mr{Tang}(\hat{f},\hat{g})$ different from $E$ is far away from this hypersurface (note also that this change of coordinate does not depend on $\varepsilon$).

In these new local coordinates,
\[\hat{f}=x_1^2 \quad\text{and}\quad \hat{g}=x_1^2u = x_1^2(u_0+x_1 \varepsilon')\]
with $u_0$ not depending on $x_1$ and $\varepsilon'$ holomorphic far from $S$ ($\varepsilon'=0$ in case $\varepsilon=0$).
The tangency locus is the union of the varieties given by the equations $x_1=0$ and $dx_1\wedge du=0$.
But $dx_1\wedge du = dx_1\wedge (du_0+x_1d\varepsilon')$ so on the exceptional divisor, the solutions of $dx_1\wedge du=0$ are the same as the solutions of $dx_1\wedge du_0=0$.
So the solutions of $dx_1\wedge du=0$ on $E$ do not depend on $\varepsilon$, thus they are $n$ simple points corresponding to the axes.

Finally, remark that $dx_1\wedge du=0$ is given by $n-1$ equations so its solution set is of dimension at least $1$.
Each point $p$ solution of these equations on $E$ then gives rise to a set $T_p$ of dimension at least $1$, but $T_p\cap E = \{p\}$ so that $\mr{dim}(T_p) = 1$.
The fact that $p$ is a simple point means that $T_p$ is a simple smooth curve intersecting $E$ transversally.
Hence, before blowing up, there were $n$ simple smooth tangency curves tangent to the ones between $q_f$ and $q_g$, which in addition implies that $\mr{Tang}(f,g)$ is diffeomorphic to $\mr{Tang}(q_f,q_g)$.

\end{proof}

Even better :

\begin{prop}
\label{prop_locus2}
There exists a diffeomorphism $\phi$ which conjugates $\mr{Tang}(f,g)$ with $\mr{Tang}(q_f,q_g)$ and $f$ with $q_f$.
\end{prop}

\begin{proof}
If we suppose that $f$ is quadratic and $q_g$ diagonal, it is enough to find $\phi$ which conjugates $\mr{Tang}(f,g)$ with $\mr{Tang}(q_f,q_g)$ and preserves $f$: $f\circ \phi=f$.
Call $D_n$ the $x_n$-axis and $T_n$ the tangency curve tangent to $D_n$.
It is sufficient to find a diffeomorphism $\phi$ preserving $f$ and fixing the points of $\{x_n=0\}$ such that $\phi(D_n) = T_n$.
Indeed, applying such a $\phi$ transforms $T_n$ into $D_n$, but if $\tilde{\phi}$ is a similar diffeomorphism obtained by exchanging tho roles of $x_n$ and $x_{n-1}$, applying $\tilde{\phi}$ transforms (the new) $T_{n-1}$ into $D_{n-1}$ and stabilizes $D_n$.
We can repeat this for each $T_j$ to obtain a diffeomorphism preserving the fibers of $f$ and conjugating the tangency loci.

The curve $T_n$ is tangent to $D_n$ so that it has equations $x_i = x_n^2 \alpha_i(x_n)$ ($i = 1,\ldots, n-1$).
We can then search $\phi$ in the form
\[\phi(x_1,\ldots,x_n) = (x_1-x_n^2 \alpha_1(x_n),\ldots,x_{n-1}-x_n^2 \alpha_{n-1}(x_n),(1+u)x_n)\]
where $u$ is an unknown holomorphic function.
The condition that $\phi$ preserve $f$ can be written
\[\sum_{i\leq n}{x_i^2} - 2x_n^2\sum_{i<n}{x_i \alpha_i(x_n)} + x_n^4\sum_{i<n}{\alpha_i(x_n)^2} + 2x_n^2u + x_n^2 u^2 = \sum_{i\leq n}{x_i^2},\]
that is
\[2u+u^2 = 2\sum_{i<n}{x_i \alpha_i} - x_n^2\sum_{i<n}{\alpha_i^2}.\]
The implicit function theorem then gives a holomorphic solution $u\in\mC{n}$ which in turn gives the desired diffeomorphism $\phi$ (note that $\phi(x_1,\ldots,x_{n-1},0) = (x_1,\ldots,x_{n-1},0)$).
\end{proof}

\begin{prop}
\label{prop_radical}
If $(f,g)$ is an $\ms{R}$-generic pair of Morse functions then the tangency ideal $I(f,g)$ is radical.
\end{prop}

\begin{proof}
Suppose that $f=\sum{x_i^2}$, $q_g=\sum{\lambda_i x_i^2}$ and that $T:= \mr{Tang}(f,g)$ is the union of the axes.
Write $df\wedge dg = \sum_{i<j}{h_{ij}dx_i\wedge dx_j}$ with $h_{ij} = 4(\lambda_j - \lambda_i)x_ix_j + O({\mC{n}}^3)$.
The ideal of functions vanishing on $T$ is $\langle x_ix_j\rangle$ so $\langle h_{ij}\rangle \subset \langle x_ix_j\rangle$ and we need to show that $\langle h_{ij}\rangle = \langle x_ix_j\rangle$.

Introduce $N=\frac{n(n-1)}{2}$ and the vectors $H=(h_{ij})_{i<j}\in(\OC{n})^{N}$ and $X=(x_ix_j)_{i<j}\in(\OC{n})^{N}$.
Note that $h_{ij}-4(\lambda_j-\lambda_i)x_ix_j \in \mC{n}\langle x_ix_j\rangle$ so that there is a matrix $A$ with coefficients in $\OC{n}$ such that $H=AX$.
Note also that $A=\Lambda+B$ where $\Lambda=diag(4(\lambda_j-\lambda_i))$ is invertible and $B$ has coefficients in $\mC{n}$.
Hence, $A$ is invertible and the coefficients of the vectors $H$ and $X$ span the same ideal.
\end{proof}

With these propositions, we can use the key lemma to conclude the $\ms{R}$-classification of pairs of Morse functions:

\begin{thm}
\label{thm_functions}
Let $(f_0,g_0)$ and $(f_1,g_1)$ be two $\ms{R}$-generic pairs of Morse functions on $(\C{n},0)$.
Suppose that we can number the tangency curves $T_j^i$ ($j=1,\ldots,n$ and $i=0,1$) in such a manner that the pairs of Morse functions $(f_i\vert_{T_j^i},g_i\vert_{T_j^i})$ are conjugated under the action of $\mr{Diff}(T_j^0,T_j^1)$ on the right.
Then there is a diffeomorphism $\varphi$ such that $( f_0\circ \varphi,  g_0\circ \varphi) = (f_1,g_1)$.
\end{thm}

\begin{proof}
By Proposition \ref{prop_locus2} we can suppose that $f_0 = f_1 = q_f$ and that the tangency loci for both couples are the same.
Then by hypothesis, $(f,g_0) = (f,g_1)$ in restriction to each tangency curve.
Since the ideals $I(f,g_0)$ and $I(f,g_1)$ are radical by proposition \ref{prop_radical}, this means that $I(f,g_0)=I(f,g_1)$ and $g_1-g_0\in I(f,g_0)$.
The proof is then completed by the lemma \ref{key_lemma}.
\end{proof}

In particular, we obtain:

\begin{cor}
An $\ms{R}$-generic pair of Morse functions $(f,g)$ is $\ms{R}$-conjugated to its quadratic parts if and only if $f$ and $g$ are $\C{}$-proportional on each tangency curve.
\end{cor}

\begin{rmq}
\label{rq_existence}
Given $n$ smooth curves $T_j$ whose tangents at $0$ span $\C{n}$ and $n$ couples $(u_j,v_j)$ of Morse functions on $T_j$, there exists a pair of Morse functions having $T_j$ as tangency curves and equal to $(u_j,v_j)$ on $T_j$.
Indeed, we can suppose that $T_j$ is the $x_j$-axis so that we can take $f(x_1,\ldots,x_n) = \sum{u_j(x_j)}$ and $g = \sum{v_j(x_j)}$.

Hence, since $f$ can be normalized, the moduli space for generic couples of Morse functions is given by the set of generic non-ordered $n$-uples $(v_1,\ldots,v_n)$ of germs of Morse functions on $\Cg{}$ modulo the relation $(v_1,\ldots,v_n)\sim (v_1\circ (\pm id),\ldots, v_n\circ (\pm id))$, the signs $\pm$ being independent.
\end{rmq}

Note also the corollary:

\begin{cor}
Let $(f_0,g_0)$ and $(f_1,g_1)$ be two $\ms{R}$-generic pairs of Morse functions on $(\C{n},0)$. 
If these pairs are topologically conjugated, they are analytically conjugated.
\end{cor}

\begin{proof}
First, note that the tangency points between $f_0$ and $g_0$ are given by the points where the Milnor number of $g_0$ restricted to a leaf of $f_0$ is greater or equal to $1$.
This characterization of the tangency points shows that a topological conjugacy between both couples respects the tangency curves.

As a consequence the restrictions of the couples $(f_i,g_i)$ to each tangency curve are topologically conjugated, and for each tangency curve $C$ there exists an homeomorphism $\phi$ of $C$ such that $l_0\circ \phi = l_1$ for $l=f,g$ on $C$.
For coordinates $z,w$ of $C$ such that $f_0(z)=z^2$ and $f_1(w)=w^2$, this equation writes $\phi(z)^2 = w^2$ so that $\phi(z) = \pm w$.
This shows that $\phi$ is holomorphic and each couples $(f_i,g_i)\vert_{T^i_j}$ are conjugated under the action of $\mr{Diff}(T_j^0,T_j^1)$ on the right.

Theorem \ref{thm_functions} can then be applied.
\end{proof}

\begin{rmq}
There is also a link between formal and analytical conjugacy: Artin's approximation theorem shows that if two pairs of germs of Morse functions are formally conjugated, they are also analytically conjugated.
\end{rmq}

\section{Pairs of Morse foliations}
\label{sec_foliations}

As stated in the introduction, the classification of pairs of Morse foliations up to diffeomorphism is equivalent to the $\ms{F}$-classification of pairs of Morse functions.
We say that a pair of Morse foliations $(\F,\G)$ is $\ms{F}$-generic if it has a pair of first integrals $(f,g)$ which is $\ms{R}$-generic. 

The invariants $(f_i\vert_{T_j^i},g_i\vert_{T_j^i})$ modulo conjugacy on the right are now only defined modulo conjugacy on the right and on the left.
First, these new invariants can be re-written in terms of involutions: on $(\C{},0)$, the data of a Morse function modulo conjugacy on the left is equivalent to the data of an involution via $f\mapsto i_f$ where $i_f$ is the function which associates to $x$ the other solution of $f(i_f(x))=f(x)$.

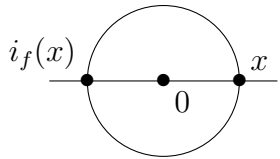
\begin{center}
\begin{tikzpicture}
\draw (0,0) circle (1);
\draw (-1.5,0) -- (1.5,0);
\draw (0,0) node {$\bullet$};
\draw (0,0) node[below right] {$0$};
\draw (1,0) node {$\bullet$};
\draw (1,0) node[above right] {$x$};
\draw (-1,0) node {$\bullet$};
\draw (-1,0) node[above left] {$i_f(x)$};
\end{tikzpicture}
\end{center}

But some information is lost in the process of considering the invariants modulo conjugacy on the left : for every pair of curves $C_1, C_2$ transverse to $\F$ and $\G$ and passing through the origin we can consider the holonomy transports $\varphi^{\F}_{12},\varphi^{\G}_{12}$ from $C_1$ to $C_2$ following the leaves of $\F$ or $\G$ :

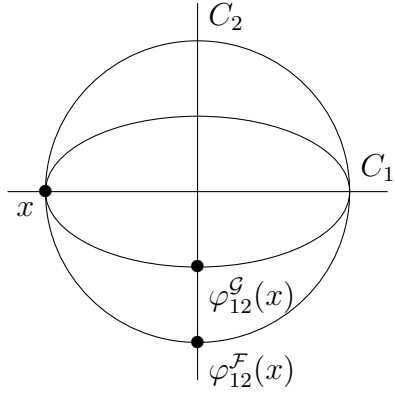
\begin{center}
\begin{tikzpicture}
\draw (0,0) circle (2);
\draw (0,0) ellipse (2 and 1);
\draw (-2.5,0) -- (2.5,0);
\draw (0,-2.5) -- (0,2.5);
\draw (2,0) node[above right] {$C_1$};
\draw (0,2) node [above right] {$C_2$};
\draw (-2,0) node {$\bullet$};
\draw (-2,0) node [below left] {$x$};
\draw (0,-2) node {$\bullet$};
\draw (0,-2) node[below right] {$\varphi^{\F}_{12}(x)$};
\draw (0,-1) node {$\bullet$};
\draw (0,-1) node[below right] {$\varphi^{\G}_{12}(x)$};
\end{tikzpicture}
\end{center}

More precisely, we will consider the holonomy transport $\varphi_{ij}^{\F}$ and $\varphi_{ij}^{\G}$ between the tangency curves $T_i$ and $T_j$.
We see on the picture that there are two possible ways to define $\varphi_{nj}^{\F}$ and $\varphi_{nj}^{\G}$, so we have to make a choice (which is equivalent to choosing a local determination of the square root).
Put then $\varphi_{njn} = (\varphi^{\G}_{nj})^{-1}\circ \varphi^{\F}_{nj}\in \mr{Diff}(T_n)$; this function allows us to recover the pair $(f\vert_{T_j},g\vert_{T_j})$ from $(f\vert_{T_n},g\vert_{T_n})$.
Indeed, take two parametrizations $\alpha_j(t)$ and $\alpha_n(t)$ of $T_j$ and $T_n$ such that $\alpha_j = \varphi_{nj}^{\F}\circ \alpha_n$.
We want to compute $g\circ \alpha_j$, but $g(\alpha_j(t)) = g((\varphi_{nj}^{\G})^{-1}(\alpha_j(t)))$ and $\alpha_j(t)=\varphi_{nj}^{\F}(\alpha_n(t))$ so $g(\alpha_j(t)) = g(\varphi_{njn}(\alpha_n(t)))$.

Note also that the invariant $\lambda_j/\lambda_n$ can be found by taking the linear part of $\varphi_{njn}$; hence the following definition:

\begin{df}
Define the invariant of $(\F,\G)$ to be $Inv(\F,\G)=((i_f^n,i_g^n),(\varphi_{njn})_{j<n})$.
Two invariants $Inv_0, Inv_1$ are equivalent if there exists a diffeomorphism $\psi \in\mr{Diff}(T_n^0,T_n^1)$ such that $\psi^{-1}\circ Inv_1\circ \psi = Inv_0$.
\end{df}

\begin{thm}
\label{thm_foliations}
Let $(\F_0,\G_0)$ and $(\F_1,\G_1)$ be two $\ms{F}$-generic pairs of Morse foliations on $(\C{n},0)$. 
Suppose that we can number their tangency curves $T_j^i$ ($j=1,\ldots,n$ and $i=0,1$) such that their invariants $Inv(f,g)$ are equivalent.
Then $(\F_0,\G_0)$ and $(\F_1,\G_1)$ are analytically conjugated.
\end{thm}

\begin{proof}
Let $(f_i,g_i)$ be first integrals for $(\F_i,\G_i)$; we can suppose that their invariants $((i_f^n,i_g^n),(\varphi_{njn})_{j<n})$ are exactly the same and that $f_0=f_1 = \sum{x_i^2}$.
We can also compose $g_1$ with a diffeomorphism on the left in such a manner that $g_0\vert_{T^0_n} = g_1\vert_{T_n^1}$ because the involutions $i_g^n$ are the same.
Then, as shown above, $g_0$ and $g_1$ are equal on each tangency curve because the $\varphi_{njn}$ are the same.

Hence Theorem \ref{thm_functions} can be applied and the pairs $(f_i,g_i)$ are indeed conjugated.
\end{proof}

Note that for each invariant $((i_1,i_2),(\varphi_{njn})_{j<n})$ there is a pair of Morse foliations having this invariant.
Indeed, we can suppose that $i_1 = -id$, $f = \sum{x_i^2}$ and that $T_j$ is the $x_j$-axis.
Choose $g$ a Morse function on $T_n$ invariant by $i_2$ and for $p_j=(0,\ldots,0,x_j,0,\ldots,0)\in T_j$ put $g(p_j) = g(\varphi_{njn}(p_n))$ for $p_n=(0,\ldots,0,x_n)$ with $x_n = x_j$.
We thus have for each curve $T_j$ a pair of Morse functions which can be extended to $(\C{n},0)$ as seen before (in the remark \ref{rq_existence}).

In order to better understand these invariants, one can find the classification of pairs of involutions in \cite{voronin} or \cite{cm}.
In particular, we see that the formal and the analytic classification of pairs of Morse foliations are not the same, because there are some pairs of involutions that are formally but not analytically conjugated.

\section{$\ms{A}$-classification of pairs of Morse functions}
\label{sec_A_classification}

We say that an application $\Phi:\Cg{n}\rightarrow\Cg{2}$ whose components $(f,g)$ are Morse functions is $\ms{A}$-generic if the pair $(f,g)$ is $\ms{R}$-generic. 

Note that the set of such applications $\Phi$ is not stable under target diffeomorphisms (for example, the diffeomorphism $(y_1,y_2) \mapsto (y_1,y_2- \lambda_1 y_1)$ transforms $(\sum{x_i^2},\sum{\lambda_i x_i^2})$ into $(\sum{x_i^2},\sum{\mu_i x_i^2})$ with $\mu_1 = 0$).
Nevertheless, a pair of functions obtained by a target diffeomorphism from an $\ms{R}$-generic pair of Morse functions still has the same tangency locus and is still classified by its values on the tangency locus.

Throughout this section, we will carry on considering pairs of Morse functions to avoid unnecessary notations, but the results extend to pairs $\ms{A}$-equivalent to an $\ms{R}$-generic pair of Morse functions.

\begin{df}
Let $\Gamma\subset \Cg{2}$ be an irreducible curve and $\sigma_1,\sigma_2:\Cg{}\rightarrow \Gamma$ two parametrizations of $\Gamma$.
We say that the parametrized curves $(\Gamma,\sigma_1)$ and $(\Gamma,\sigma_2)$ are $\sigma$-equivalent if there is a diffeomorphism $\phi\in\Df{}$ such that $\sigma_1\circ \phi=\sigma_2$.
An equivalence class $[(\Gamma,\sigma)]$ is called a $\sigma$-curve; we define its $\sigma$-multiplicity to be the integer $n$ such that $\sigma(t) = (at^n+\ldots,bt^n+\ldots)$ with $(a,b)\neq(0,0)$.
\end{df}

If the parametrization is clear from the context, we may omit to mention it.

\begin{rmq}
A $\sigma$-curve $[(\Gamma,\sigma)]$ is entirely determined by $\Gamma$ and its $\sigma$-multiplicity.

A $\sigma$-curve $[(\Gamma,\sigma)]$ is of $\sigma$-multiplicity $2$ in exactly two cases: either $\Gamma$ is diffeomorphic to a curve $y^2-x^{2k+1}$ ($k\geq 1$) and $\sigma$ is a bijection or $\Gamma$ is smooth and $\sigma$ is a double cover.
The last case happens for example when $\sigma(t)=(t^2,b(t^2))$.
\end{rmq}

We saw that pairs of Morse functions are classified modulo the action of diffeomorphisms at the source only by the restrictions of $\Phi=(f,g)$ on the tangency curves $T_i$ between $f$ and $g$, ie. on the critical set of $\Phi$.
Said another way, the classification is given by the functions $\Phi_{\vert T_i}$ with diffeomorphisms at the source acting as reparametrization, that is by the $\sigma$-curves $\Phi(T_i)\subset \Cg{2}$.

Each of these $\sigma$-curves has $\sigma$-multiplicity $2$ at the origin and has the line $(t^2, \lambda_i t^2)$ as tangent cone if $f_{\vert T_i}(t)=t^2+\ldots$ and $g_{\vert T_i}(t) = \lambda_i t^2 + \ldots$

Thus the result is the following:

\begin{thm}
Two $\ms{A}$-generic pairs of Morse functions $\Phi_1$ and $\Phi_2$ are $\ms{A}$-conjugated if and only if the set of $\sigma$-curves $\{\Phi_1(T_i^1)\}_{i\leq n}$ and $\{\Phi_2(T_i^2)\}_{i\leq n}$ are conjugated by a diffeomorphism of $\Cg{2}$.

Moreover, for each set of $n$ $\sigma$-curves $\{C_i\}$ in $\Cg{2}$ with $\sigma$-multiplicity $2$ and distinct tangent cones, there exists an application $\Phi:\Cg{n} \rightarrow \Cg{2}$ whose components are Morse functions for which $C_i = \Phi(T_i)$.
\end{thm}

\begin{rmq}
A diffeomorphism $\psi$ of $\Cg{2}$ conjugates two families of $\sigma$-curves $([C_i^1,\sigma_i^1])$ and $([C_i^2,\sigma_i^2])$ if and only if for each $i$, the $\sigma$-curves $C_i^1$ and $C_i^2$ have the same multiplicity and $\psi$ conjugates the families of curves $(C_i^1)$ and $(C_i^2)$.
\end{rmq}

\begin{proof}
Clearly, if two pairs are conjugated by source and target diffeomorphisms, their critical sets are conjugated at the source, so the images of the critical sets are conjugated at the target.

Conversely, suppose that for two generic pairs $\Phi_j=(f_j,g_j)$ there exists a diffeomorphism $\psi\in \Df{2}$ conjugating the sets of $\sigma$-curves $\{\Phi_j(T_i^j)\}_{i\leq n}$.
Then we can suppose these sets to be equal, which means that for the right numbering of the tangency curves, the $\sigma$-curves $\Phi_1(T_i^1)$ and $\Phi_2(T_i^2)$ are equal for each $i$.
This gives for every $i$ a diffeomorphism $\varphi_i: T_i^1 \rightarrow T_i^2$ such that $\Phi_{1\vert T_i^1} = \Phi_{2\vert T_i^2}\circ \varphi_i$.

We can then conclude with theorem \ref{thm_functions}.
\\

For the realization part of the theorem, take $n$ $\sigma$-curves $C_i$ in $\Cg{2}$ with $\sigma$-multiplicity $2$ and distinct tangent cones.
Note first that we can suppose that no curve has an axe as tangent cone so that these $\sigma$-curves can be parametrized by $\sigma_i(t) = (t^2, \lambda_i t^2 + O(t^3)) =: (u_i(t),v_i(t))$ with $\lambda_i\neq 0$.
But these curves are the images of the critical locus of the pair $(\sum{u_i(x_i)},\sum{v_i(x_i)})$ which is $\ms{A}$-generic because $\lambda_i \neq \lambda_j$ if $i\neq j$ and this concludes the proof.
\end{proof}

\section{Quotients of Morse functions}
\label{sec_quotients}

Next, consider meromorphic functions $h=g/f$ with $f,g\in \OC{n}$ Morse functions satisfying the genericity condition.
We want to classify these functions up to diffeomorphism at the source.

First, consider the critical locus of $h$ : it is given by the zeroes of $\omega = gdf-fdg$, which contain the indeterminacy locus $\{f=0\}\cap\{g=0\}$.
Note that when $f=\sum{x_i^2}$ and $g=\sum{\lambda_i x_i^2}$, the critical locus contains not only $\{f=0\}\cap\{g=0\}$ but also the union of the axes.
We begin by showing that after a generic perturbation, only the indeterminacy locus remains.
Denote by $I(\omega)$ the ideal spanned by the components of $\omega$.

We say that a pair of Morse functions is $\ms{Q}$-generic if it is diffeomorphic to $(\sum{x_i^2}, \sum{\lambda_i x_i^2+\alpha_ix_i^3+O(\mf{m}^4)})$ with $\lambda_i\neq \lambda_j$ and $\alpha_i\neq 0$; we say that a quotient $g/f$ is $\ms{Q}$-generic if the pair $(f,g)$ is $\ms{Q}$-generic.

\begin{lem}
For a $\ms{Q}$-generic pair of Morse functions $(f,g)$, the ideal $I(\omega)$ contains $\langle f,g\rangle \cdot \mC{n}^4$.
\end{lem}

\begin{proof}
For simplicity, denote $\mf{m}=\mC{n}$.
By theorem \ref{thm_functions} we can suppose that $f=\sum{x_i^2}$ and $g=\sum{u_i(x_i)}$.
The genericity hypothesis thus means that $u_i(x_i) = \lambda_i x_i^2 + \alpha_i x_i^3+O(x_i^4)$ with $\alpha_i \neq 0$.
If we write $\omega = \sum{ \omega_i dx_i}$, the coefficient $\omega_i$ is
\[
\omega_i = 2\sum_{j\neq i}{(\lambda_j-\lambda_i)x_ix_j^2} + O(\mf{m}^4)
\]
so that $\omega_i = 2x_i (g - \lambda_i f) + O(\mf{m}^4)$.
Hence the equalities $x_j \omega_i- x_i \omega_j = 2x_ix_j(\lambda_j-\lambda_i)f+O(\mf{m}^5)$ and $\lambda_jx_j \omega_i - \lambda_ix_i \omega_j = 2x_ix_j(\lambda_j-\lambda_i)g+O(\mf{m}^5)$.
As a consequence, for each monomial $m$ of degree $4$ except $m=x_k^4$ and each $l=f,g$, we have $m l \in I(\omega)+\mf{m}^7$.
Furthermore,
\begin{align*}
\frac{1}{2}\sum_i{x_i \omega_i} &= \sum_i{ \frac{1}{2}x_i \left(g \partial_{x_i}f - f \partial_{x_i}g \right)}\\
&= g\sum_i{\frac{1}{2}x_i \partial_{x_i}f} - \frac{1}{2}f\sum_i{x_i \partial_{x_i}g}\\
&= gf - \frac{1}{2}f\sum_i{x_i \partial_{x_i}g}\\
&= f \left(g - \sum_i{\frac{1}{2}x_i \partial_{x_i}g} \right)\\
&= f \left(\frac{-1}{2}\sum_i{\alpha_i x_i^3} + O(\mf{m}^4) \right).\\
\end{align*}
Thus, $x_i\sum{x_j \omega_j} = \beta_i x_i^4 f + \sum_{j\neq i}{\beta_j x_jx_i^3 f} + O(\mf{m}^7)$ for some non-zero coefficients $\beta_k$, and $x_i^4f \in I(\omega) + \mf{m}^7$.

A similar computation shows that $x_i^4g\in I(\omega) + \mf{m}^7$; so for each monomial $m$ of degree $4$ and each $l=f,g$, we have $ml\in I(\omega)+\mf{m}^7$.
In fact, $ml$ belongs to the ideal $I(\omega) + \langle f,g \rangle\cdot \mf{m}^5$ because $I(\omega)$ is obviously a subset of $\langle f,g \rangle$.
It immediately follows that for each index $k\geq 4$, each monomial $m$ of degree $k$ and each $l=f,g$, $ml\in I(\omega)+\langle f,g\rangle \cdot \mf{m}^{k+1}$.
This means that $ml$ formally belongs to the ideal $I(\omega)$ hence by flatness, $\langle f,g\rangle\cdot \mf{m}^4\subset I(\omega)$.
\end{proof}

\begin{rmq}
Note that the proof is still valid for $1$-parameter families $(f_t),(g_t)$ with fixed $3$-jets.
Indeed, we can show in the exact same way that $ml\in I(\omega) + \mf{m}^7$ for each monomial $m$ in $x$ of degree $4$ and $l=f,g$, the only difference is that $f, g$ and $\omega$ depend on $t$ (here $\mf{m}$ is still $\langle x_1,\ldots,x_n\rangle$).

Note also that for $1$-parameter families $(f_t),(g_t)$ with fixed $3$-jets, being a $\ms{Q}$-generic pair of Morse functions for each $t\in\C{}$ is equivalent to being a $\ms{Q}$-generic pair of Morse functions for $t=0$ because the genericity only depends on the $3$-jets.

We thus obtain the following:
\end{rmq}

\begin{lem}
\label{lemma_meromorphic}
Consider two functions $f,g\in \mc{O}(t,x_1,\ldots,x_n)$ defined in a neighborhood of $\C{}_t\times\{0\}\subset\C{}_t\times\C{n}_x$ with $3$-jets independent of $t$.
Suppose that $(f(t,\cdot),g(t,\cdot))$ is a $\ms{Q}$-generic pair of Morse functions for each $t$.
Consider $\omega_x= gd_xf-fd_xg$ and $\mf{m}=\langle x_1,\ldots,x_n\rangle$, then $\langle f,g\rangle \mf{m}^4\subset I(\omega_x)$.
\end{lem}

\begin{thm}
\label{thm_meromorphic}
Let $h_0$ and $h_1$ be $\ms{Q}$-generic quotients of Morse functions with $h_i=g_i/f_i$.
Suppose that we have equalities between the $3$-jets: $j^3f_0=j^3f_1$ and $j^3g_0=j^3g_1$.
Then there exists a diffeomorphism $\varphi\in \Df{n}$ such that $h_0\circ \varphi = h_1$. 
\end{thm}

\begin{proof}
By theorem \ref{thm_functions}, we can suppose that $g_k=\sum_i{u^k_i(x_i)}$ and $f_k=\sum_i{x_i^2}$ with $u^k_i(x) = \lambda_i x^2 + \alpha_i x^3 + \varepsilon^k_i$ with $\alpha_i\neq 0$ and $\varepsilon^k_i\in \mC{n}^4$.
Set for $t$ in a neighborhood of $[0,1]$ in $\mb{C}$ $f(t,\cdot) = f_t = f_0= f_1$, $g(t,\cdot) = g_t = g_0 + t(g_1-g_0)$, $h(t,\cdot) = h_t=g_t/f_t$ and $\omega = g df - f dg = \omega_x + r dt$.

Note that $r = -f \partial_{t}g \in \langle f,g \rangle \mf{m}^4$ and that by lemma \ref{lemma_meromorphic}, this implies $r\in I(\omega_x)$.
We can then find a vector field $X = \sum_i{X_i \partial_{x_i}} + \partial_{t}$ such that $\omega(X)=0$ (note that $X_i\in\langle x_1,\ldots,x_n\rangle$ because there is no linear relation with constant coefficients between the leading terms of the components of $\omega_x$).
But this means that $h$ is constant along the trajectories of $X$ so that the flow $\varphi_s(x,t)$ of $X$ (which is defined on a neighborhood of $\{0\}\times[0,1]$) gives a diffeomorphism $\varphi:x\mapsto \varphi_1(x,0)$ such that $h_0\circ \varphi = h_1$ on $\Cg{n}$.
\end{proof}

\begin{cor}
Let $h$ be a $\ms{Q}$-generic quotient of Morse functions.
There exists $\lambda_i,\alpha_i \in \mb{C}^*$ such that $h$ is diffeomorphic to
\[
\frac{\sum_{i}{\lambda_i x_i^2+\alpha_i x_i^3}}{\sum_i{x_i^2}}.
\]
\end{cor}

\begin{rmq}
Since the latter form is stable under homotecies, we can even suppose that $\alpha_1 = 1$.
\end{rmq}

\section{Restriction of a Morse function to a quadratic cone}
\label{sec_restrictions}

In this section, we want to study restrictions of Morse functions $g$ to a "quadratic cone" (ie. an hypersurface $\{f=0\}$ with $f$ also a Morse function).

\begin{rmq}
We can see by a cohomological argument that each function and each diffeomorphism defined on a quadratic cone extends to $\Cg{n}$ (respectively to a function or a diffeomorphism of $\Cg{n}$).
Thus, studying functions on a quadratic cone up to diffeomorphism of the cone is the same as studying functions of $\Cg{n}$ in restriction to a quadratic cone up to diffeomorphisms of $\Cg{n}$ fixing the cone.
\end{rmq}

\begin{thm}
Let $f$, $g_0$ and $g_1$ be three Morse functions with $(f,g_i)$ $\ms{R}$-generic pairs and equalities between the $2$-jets $j^2g_0=j^2g_1$.
Then there is a diffeomorphism $\varphi$ such that $f\circ \varphi=f$ and $g_0\circ \varphi = g_1$ in restriction to $\{f=0\}$.
\end{thm}

\begin{proof}
Let $g_t = g_0 + t(g_1-g_0)$.
We want to find a diffeomorphism $\varphi$ such that $f\circ \varphi=f$ and $g_0\circ \varphi - g_1 \in \langle f\rangle$; we will use Moser's path method to find it as the flow of a vector field $X=\sum{X_i \partial_{x_i}}+\partial_{t} $ such that $X\cdot g\in \langle f\rangle$ and $X\cdot f=0$.
Note that we can find $X$ verifying $X\cdot g = X\cdot f = 0$ as soon as $\partial_{t}g\in I(f,g)$, so that we can find $X$ as sought as soon as $\partial_{t}g\in \langle f\rangle + I(f,g)$.
Remark that the components of $X-\partial_{t}$ will cancel on the $t$-axis because there is no linear relation with constant coefficients between $f$ and the components of $df\wedge dg$.

We saw in the proof of proposition \ref{prop_radical} that $I(f,g)=\langle x_ix_j + \ldots\rangle$, but $x_i^3$ is equal to $x_if$ modulo the ideal $I(f,g) + \mC{n}^4$ so that each monomial of degree $3$ belongs to $\langle f\rangle + I(f,g) + \mC{n}^4$.
Thus, the inclusion $\mC{n}^3\subset \langle f\rangle + I(f,g)$ holds so that $\partial_{t}g\in\langle f \rangle + I(f,g)$ and the proof is complete.
\end{proof}

\begin{rmq}
Note also that $g$ and $g+\lambda f$ represent the same function on $\{f=0\}$ so that we obtain the following:
\end{rmq}

\begin{cor}
Given a Morse function $f$, each Morse function $g$ such that the pair $(f,g)$ is $\ms{R}$-generic is diffeomorphic in restriction to $\{f=0\}$ to a quadratic function $\sum_{i=1}^{n-1}{\lambda_i x_i^2}$.
\end{cor}

\section{Applications of the Key Lemma}
\label{sec_generalisation}

The key lemma can be used in a very general setting for the $\ms{R}$-classification of pairs of functions: although the hypotheses might seem strong, they are in fact necessary.
For example, it can be applied to rediscover the $\ms{R}$-classification of folds. 

\begin{df}
Define a fold to be a pair of functions $f,g : (\C{n},0)\rightarrow (\C{},0)$ such that $f$ is regular and $\mr{Tang}(f,g)$ is a simple smooth curve transverse to $\{f=0\}$.
\end{df}

\begin{thm}
\label{thm_folds}
Let $(f,g)$ be a fold on $\Cg{n}$.
There exists a unique function $\varphi\in \mc{O}(\C{},0)$ and a set of coordinates $(x_i)$ such that $f=x_1$ and $g=\varphi(x_1)+\sum_{i>1}{x_i^2}$.
\end{thm}

\begin{proof}
We can suppose without loss of generality that $f=x_1$ and that $\mr{Tang}(f,g)$ is the $x_1$-axis.
This means that $I(f,g) = \langle \partial_{x_i}g\rangle_{i>1} = \langle x_2,\ldots,x_n\rangle$ so $g = \varphi(x_1) + q(x_2,\ldots,x_n) + \varepsilon$ with $q$ a nondegenerate quadratic form and $\varepsilon\in \langle x_2,\ldots,x_n\rangle^2 \mC{n}$.
Since $q$ is nondegenerate, we can suppose $q = \sum_{i>1}{x_i^2}$.

We want to use the key lemma in $(\C{n},0)$ for $f = x_1$, $g_0 = \varphi(x_1) + x_2^2+\ldots+x_n^2$ and $g_1 = g$.
Let us check the hypotheses: first, $g_1-g_0 = \varepsilon\in\langle x_2,\ldots,x_n\rangle^2\mC{n}$.
Then, for each $a\in\langle x_2,\ldots,x_n\rangle^2\mC{n}$, the ideal $I(f,g_0+a)$ writes $\langle x_2+\eta_2,\ldots, x_n+\eta_n\rangle$ with $\eta_i \in \langle x_2,\ldots,x_n\rangle \mC{n}$, which means that $\mr{Tang}(f,g_0+a)$ is a simple curve and the ideal $I(f,g_0+a)$ is radical.
So the hypotheses $I(f,g_0)=I(f,g_1)$ and $g_1-g_0\in I(f,g_0)$ are verified, hence the only hypothesis missing is $f$ having a singular point.

But $g_1-g_0$ cancels at order $3$ at the origin, which will allow us to use the remark \ref{rmq_lemma}.
Indeed, if we use the same notations, the fact that there is no $\C{}$-linear relation between the generators of $I(f,g_0)$ implies that the coefficients $r_{ij}$ in the decomposition $g_1-g_0 = \sum{r_{ij}h_{ij}}$ cancel on $\{0\}\times[0,1]$. 
The lemma can thus be applied and the couples $(f,g_0)$ and $(f,g_1)$ are diffeomorphic.

Last, the function $\varphi$ is entirely determined by the equality $\varphi\circ f = g$ on $\mr{Tang}(f,g)$.
\end{proof}

A first corollary is the classification of regular folds as foliations (ie. the $\ms{F}$-classification):

\begin{cor}
Let $(\F,\G)$ be a pair of foliations on $\Cg{n}$ given by a fold $(f,g)$ with $g$ regular.
Then $(\F,\G)$ is diffeomorphic to the pair of foliations given by the first integrals $(x_1,x_1+\sum_{i>1}{x_i^2})$.
\end{cor}

\begin{proof}
We can suppose that $(f,g)$ are as in the conclusion of the theorem \ref{thm_folds}.
The hypothesis that $g$ be regular means that $\varphi$ is a diffeomorphism.
In the variables $(\varphi(x_1), x_2,\ldots,x_n)$ the pair $(\F,\G)$ is in the right form.
\end{proof}

We can also use this to obtain the $\ms{R}$-classification of generic pairs $(f,g)$ with $f$ regular and $g$ a Morse function: this is exactly when $\varphi$ is a Morse function.
In the case of $\ms{F}$-equivalence, we obtain the normal form $(x_1,\sum_{i\geq 1}{x_i^2})$.
\\

We could also study pairs $(f,g)$ of the form $(x^3+y^2+z^2,\lambda x^2+\mu y^2+ \nu z^2 + \ldots)$, but in this case the tangency ideal $I(f,g)$ will again be radical and this case will be similar to the case of pairs of Morse functions.
\\

The lemma \ref{key_lemma} can also be applied for more complicated cases, like for example when the ideal $I(f,g)$ is not radical.
To illustrate this, note that if we take $f=x^3+y^2+z^2$ and $g=\lambda x^3+\mu y^2+\nu z^2$ with $\lambda\neq\mu\neq\nu\neq 0$, the tangency ideal is $I(f,g) = \langle x^2y,x^2z,yz\rangle$ and corresponds to $D_x\cup 2D_y\cup 2D_z$ with $D_l$ the $l$-axis.
Let us classify pairs of functions that "look like" this pair.
First, recall the following:

\begin{prop}
Let $f$ be a function on $\Cg{3}$ having a singular point with Milnor number $2$ at the origin; then in a right set of coordinates, $f(x,y,z) = x^3+y^2+z^2$.
\end{prop}

\begin{proof}
Since the Milnor number of $f$ is $2$, the hessian matrix of $f$ at $0$ is of rank $2$ and in the right set of coordinates, it can be written $diag(0,2,2)$.
Then $f(x,y,z) = y^2+z^2 + \varepsilon$ with $\varepsilon\in \mf{m}^3$ and $f$ can be seen as a deformation of $f(0,\cdot,\cdot)$ which has a non-degenerate singular point at $0$.
By the parametrized Morse lemma, there exists a function $\varphi$ and a set of coordinates such that $f(x,y,z) = \varphi(x) + y^2+z^2$.

Since the Milnor number of $f$ is $2$, $\varphi$ is diffeomorphic to $x^3$ and changing the coordinates once more, we can write $f(x,y,z) = x^3+y^2+z^2$.
\end{proof}

So in fact we are interested in pairs $(f,g)$ of functions with Milnor number $2$, having hessians $H(f)$, $H(g)$ which can be simultaneously diagonalized with the $0$ in the same spot.
For such functions, we can then suppose that 
\begin{equation}
\label{form_cusps}
f=x^3+y^2+z^2\quad \text{and}\quad g=\lambda x^3 + \mu y^2+\nu z^2 + \varepsilon
\end{equation}
with $\varepsilon\in\mf{m}^3$ which has no component in $x^3$.

The tangency locus might not be diffeomorphic to the union of one simple curve and two double curves: the double curves might split.
For example for $f=x^3+y^2+z^2$ and $g=x^3+\mu y^2+\nu z^2 + x^2y$, the $y$-axis splits into two curves tangent respectively to the $y$-axis and to the line $\{z=0=3(\mu-1)x-2y\}$.
Let us assume the double curves don't split.
We will call such a pair $(f,g)$ an exceptional pair of $3$-dimensional cusps (or an exceptional pair of cusps because we only deal with the $3$-dimensional ones in this example).

\begin{prop}
If $(f,g)$ is an exceptional pair of cusps written as in \eqref{form_cusps}, then $\mr{Tang}(f,g)$ is tangent and diffeomorphic to the union of the axes.
Moreover, the tangency curve tangent to the $x$-axis is tangent at order $2$ with the $x$-axis.
\end{prop}

\begin{proof}
The proof is very similar to that of the proposition \ref{prop_locus1}.
We have:
\begin{align*}
df\wedge dg &= \left( 3x^2(2\mu y+\partial_{y}\varepsilon)-2y(3\lambda x^2+\partial_{x}\varepsilon) \right)dx\wedge dy +\\
& \left( 3x^2(2\nu z+\partial_{z}\varepsilon)-2z(3\lambda x^2 + \partial_{x}\varepsilon) \right)dx\wedge dz +\\
& \left( 2y(2\nu z+\partial_{z}\varepsilon)-2z(2\mu y + \partial_{y}\varepsilon) \right)dy\wedge dz.
\end{align*}

Introduce the $3$-jet $\varepsilon_3$ of $\varepsilon$.
Remember that $\varepsilon\in \mf{m}^3$ so that after projectivization, the tangency cone is given in $\mb{P}^2(\C{})$ by the system of equations
\[
\left\{\begin{aligned}
&4(\nu-\mu)yz\\
&2\left(3(\mu-\lambda)x^2-\partial_{x}\varepsilon_3 \right)y\\
&2\left(3(\nu-\lambda)x^2-\partial_{x}\varepsilon_3 \right)z\\
\end{aligned}\right.
\]

Therefore, the hypothesis that the tangency curves don't split implies that $\partial_{x}\varepsilon_3=0$, ie. $\varepsilon_3\in\langle y^3,y^2z,yz^2,z^3\rangle$, and in this case each component of the tangency locus is tangent to an axe.

To show that $\mr{Tang}(f,g)$ is indeed diffeomorphic to the union of the axes, blow up the origin.
If we blow up in the direction $y$ (it will be the same in the direction $z$), we obtain $\tilde{f}=y^2(yx^3+1+z^2)$ and $\tilde{g}=y^2(\lambda yx^3 + \mu + \nu z^2 + y\tilde{\varepsilon})$.
Hence, near the point $(0,0,0)$ we can make the change of coordinate $y\mapsto \sqrt{\tilde{f}}$ and obtain $\tilde{f}=y^2$, $\tilde{g}=y^2u$ with
\begin{align*}
u &= u_0(z^2) + yu_1(x^3,z)+O(y^2)\\
&= \left( \mu+(\nu-\mu)z^2+O(z^4) \right) + y \left(a+(\lambda-\mu)x^3+O(z) \right) + O(y^2)
\end{align*}
with $a$ depending on $\varepsilon$.
We see that on the exceptional divisor $E$, the equation $dy\wedge du=0$ is equivalent to $dy\wedge du_0=0$, that is $z=0$: it is a line and not a point as before.
But a tangency point is a singular point of $u\vert_{\{y=y_0\}}$ for some $y_0$ so we only need to search for the singular points of $u\vert_{\{y=y_0\}}$.
Note that
\[\partial_{z}u = 2(\nu-\mu)z + O(z^2) + O(y) \quad\text{and}\quad \partial_{x}u = y(x^2(3(\lambda-\mu)+\ldots)+O(y))\]
so, for $y_0$ near $0$, there are two singular points tending to $0$ as $y_0$ tends to $0$.
These points form a set intersecting $E$ at $0$ with multiplicity $2$, and our hypothesis that the double curve hasn't split shows that these two points are equal and that there is a double tangency curve tangent to the $y$-axis.

Next, if we blow up in the direction $x$, we obtain $\tilde{f}=x^2(x+y^2+z^2)$ and $\tilde{g}=x^2(\lambda x + \mu y^2 + \nu z^2 + x \tilde{\varepsilon})$.
We cannot make the desired change of coordinate near $(0,0,0)$ so blow up once more:
\[\hat{f}=x^3(1+x y^2+x z^2)\quad\text{and}\quad \hat{g}=x^3(\lambda + \mu xy^2+\nu xz^2 + x\hat{\varepsilon})\]
and call $E$ the exceptional divisor corresponding to the last blow up.
There we can make the change $x\mapsto \hat{f}^{1/3}$ and get $\hat{f}=x^3$, $\hat{g}=x^3u$.
Here, $u$ is $(\lambda+\mu xy^2+\nu xz^2 + x\hat{\varepsilon})(1+xy^2+xz^2)^{-1}+O(x^2)$ but remember that $\varepsilon_3\in\langle y^3,y^2z,yz^2,z^3\rangle$ so $\tilde{\varepsilon}$ has no constant term and $\hat{\varepsilon}$ is in fact divisible by $x$.
Thus
\[u=u_0 + O(x^2) \quad\text{with}\quad u_0=(\lambda+\mu xy^2 + \nu xz^2)(1+xy^2+xz^2)^{-1}.\]
But $du_0$ is null in restriction to $E$ so take away the constant term by replacing $\hat{g}$ by $\hat{g}-\lambda\hat{f}$: we can do this because $d\hat{f}\wedge d\hat{g} = d\hat{f}\wedge d(\hat{g}-\lambda\hat{f})$.
Finally, $\hat{g}-\lambda\hat{f} = x^4 v$ with $v = v_0 + O(x)$ and $v_0 = v_0(y,z)$ is a Morse function at $0$.
Thus the equation $df\wedge dv=0$ has a set of solution of dimension $1$ but only one solution on $E$: the solution of $df\wedge dv_0=0$.
Then there is one tangency curve tangent to the $x$-axis.
Note that this tangency curve wasn't separated from the $x$-axis after $2$ blow-ups so they are tangent at order (at least) two.

Finally, we need to show that there are no other tangencies.
We already studied the blow-up in the directions $y$ and $z$, so we only need to study the tangency locus near the $x$-axis.
Remember that near the $x$-axis, we have
\[
\tilde{f}=x^2(x+y^2+z^2)\quad\text{ and }\quad \tilde{g}=x^2(\lambda x + \mu y^2 + \nu z^2 +x \tilde{\varepsilon}).
\]
So
\begin{align*}
d\tilde{f}\wedge d\tilde{g}
&= x^3 \left[6(\mu- \lambda)xy + O(\mf{m}^3)\right] dx\wedge dy\\
&+ x^3 \left[6(\nu- \lambda)xz + O(\mf{m}^3) \right]dx\wedge dz\\
&+ x^4 \left[4(\nu-\mu) yz +O(\mf{m}^3) \right] dy\wedge dz
\end{align*}
where $\mf{m}=\langle x,y,z\rangle$.
Thus, near the intersection point between the $x$-axis and the exceptional divisor, after removing the powers of $x$, we see that each component of the tangency locus is tangent to one of the axes (in the blow-up coordinates) and that there can be at most one curve tangent to each axis.
But we have already shown that there is one tangency curve tangent to the $x$-axis, and we have found two curves inside the exceptional divisor solutions to $d\tilde{f}\wedge d\tilde{g}=0$, which were tangent to the $y$-axis and the $z$-axis.
Hence there can be no other tangencies around the $x$-axis.

\end{proof}

\begin{prop}
\label{prop_locus_cusp}
If $(f,g)$ is an exceptional pair of cusps in the form \eqref{form_cusps}, there exists a diffeomorphism $\varphi$ preserving $f$ such that $\mr{Tang}(f\circ \varphi, g\circ \varphi)$ is the union of the axes.
\end{prop}

\begin{proof}
We will follow the same reasoning as in the proof of proposition \ref{prop_locus2}: call $D_l$ the $l$-axis ($l=x,y,z$) and $T_l$ the tangency curve tangent to $D_l$.
We want to show that there exists a diffeomorphism $\varphi$ fixing $\{x=0\}$ and sending $T_x$ to $D_x$.
For $l=y,z$, finding a diffeomorphism fixing $\{l=0\}$ and sending $T_l$ to $D_l$ can be done exactly as in the proposition \ref{prop_locus2}, but for $l=x$ a little change must be done:
the curve $T_x$ is tangent to $D_x$ at order $2$ so it has equations $y=x^3 \alpha_2(x)$, $z=x^3 \alpha_3(x)$.
We will search $\varphi$ in the form
\[\varphi(x,y,z) = ((1+u)x,y-x^3 \alpha_2(x), z-x^3 \alpha_3(x))\]
where $u$ is an unknown holomorphic function.
We need to have
\[(1+3u+3u^2+u^3)x^3 + y^2 + z^2 - 2x^3 \left(y \alpha_2(x) + z \alpha_3(x) \right) + x^6 \left(\alpha_2(x)^2 + \alpha_3(x)^2 \right) = x^3+y^2+z^2,\]
that is
\[3u+3u^2+u^3 = 2 \left(y \alpha_2 +z \alpha_3 \right) - x^3 \left(\alpha_2^2 + \alpha_3^2 \right).\]
The implicit function theorem gives a solution $u\in \mC{3}$ and the desired diffeomorphism $\varphi$ follows.
\end{proof}

Since the ideal is not radical, the tangency locus is not sufficient to characterize the ideal.
The following proposition gives a geometric description of the ideal; it might be interesting in other contexts because it hints at something more general: the characterization of any ideal in terms of cancellation of functions and cancellation of some differential operators on these functions.
But I couldn't find mention of such a characterization anywhere, so we only give the following special case:

\begin{prop}
\label{prop_chp_vect}
Let $(f,g)$ be an exceptional pair of cusps in the form \eqref{form_cusps} with $\mr{Tang}(f,g)$ equal to the union of the axes.
Then there is a vector field $X$ such that $X(0)=\partial_{x}$ and
\[I(f,g) = \left\{a\in\OC{3} \text{ such that } a\vert_{T_x}=a\vert_{T_y}=a\vert_{T_z}=0 \text{ and } (X\cdot a)\vert_{T_y}=(X\cdot a)\vert_{T_z}=0 \right\}.\]
Such a vector field will be said to characterize the tangency ideal.
\end{prop}

\begin{proof}
In the computations done before, we saw that $I(f,g)$ is spanned by the functions $h_1=x^2y+O(\mf{m}^4)$, $h_2=x^2z+O(\mf{m}^4)$ and $h_3=yz+O(\mf{m}^3)$
Note that the the tangent cone at $0$ of the variety $\{h_3=0\}$ is the union of the planes $\{y=0\}$ and $\{z=0\}$.
Moreover, we know by hypothesis that $h_3(T_y) = h_3(T_z) = \{0\}$ so for each $z$ near $0$, there is a unique plane tangent to $\{h_3=0\}$ at the point $(0,0,z)$.
This plane contains the direction $T_z$ so it is defined by another direction $X(z)$ which we can choose regular in $z$ with $X(0)=\partial_{x}$.
Similarly, the tangent plane to $\{h_3=0\}$ along $T_y$ is defined by a vector field along $T_y$ which we can choose so that both vector fields can be extended to a vector field $X$ on $\Cg{3}$ with $X(0) = \partial_{x}$.

Now let
\[J=\{a\in \OC{3} \text{ such that } a\vert_{T_x}=a\vert_{T_y}=a\vert_{T_z}=0 \text{ and } (X\cdot a)\vert_{T_y}=(X\cdot a)\vert_{T_z}=0 \}.\]
The set $J$ is an ideal and we first need to show that $I(f,g)\subset J$, ie. that $(X\cdot h_i)\vert_{T_l}=0$ for $i=1,2,3$ and $l=y,z$.
By construction, $(X\cdot h_3)\vert_{T_y}$ and $(X\cdot h_3)\vert_{T_z}$ are null.
Next, we know that $h_1\in \langle xy,yz,zx \rangle$ so up to changing $h_1$ by $h_1 - \sum_{i=2,3}{\lambda_i h_i}$ with $\lambda_i\in \mC{3}$, we can suppose that $h_1 = ux^2y + x \alpha(y) + x \beta(z)$ with $u$ invertible, $\alpha$ and $\beta$ in $\mC{}^3$.

The condition that the tangency curves do not split implies that when cutting the curve $T_y$ by a plane $y=y_0$, we obtain a point with multiplicity $2$.
But if $\alpha\neq 0$, then $\alpha(y_0)$ is generically invertible and $h_1(\cdot,y_0,\cdot)$ is generically regular.
The function $h_3(\cdot,y_0,\cdot)$ is also generically regular, so if $\alpha\neq 0$, we obtain a simple point; hence $\alpha=0$.
By the same reasons, $\beta=0$ and $I(f,g) = \langle x^2y,h_2,h_3 \rangle$.
Similarly, $I(f,g) = \langle x^2y,x^2z,h_3 \rangle$ and it is now clear that $I(f,g)\subset J$.

For the converse, we will show that $(x^2y, x^2z, h_3)$ generate $J$: suppose $a\in J$ and $P$ is his leading homogeneous polynomial (and let $k+1$ be his degree).
Since $J\subset \langle xy,yz,zx \rangle$, $P$ has no term in $l^{k+1}$ for $l=x,y$ or $z$.
The only terms that are not spanned by the leading coefficients of $x^2y,x^2z$ or $h_3$ are the $xl^k$ for $l=y,z$.
But if $X=(1+a_1)\partial_{x}+a_2 \partial_{y}+ a_3 \partial_{z}$, then $X\cdot xy^k = (1+a_1)y^k + ka_2xy^{k-1}$ is not nul on $T_y$: there can't be such a term in $P$.
Therefore $(x^2y,x^2z,h_3)$ generate $J$ and $I(f,g)=J$.
\end{proof}

\begin{prop}
If $(f,g)$ is an exceptional pair of cusps in the form \eqref{form_cusps}, there exists a diffeomorphism $\varphi$ preserving $f$ such that $I(f\circ \varphi, g\circ \varphi) = \langle x^2y,x^2z,yz \rangle$.
\end{prop}

\begin{proof}
By the proposition \ref{prop_locus_cusp} we can suppose that the tangency locus is the union of the axes.
By the proposition \ref{prop_chp_vect} we can find a vector field $X$ such that $X(0)= \partial_{x}$ characterizing the tangency ideal.
We want to transform $X$ into $\partial_{x}$ using a diffeomorphism $\varphi$ preserving $f$ and the coordinate axes.

As before we will construct $\varphi$ in two steps by transforming the vector field first on the $y$-axis and then on the $z$-axis.
We will search the first diffeomorphism in the form $\varphi_1(x,y,z) = (x+yxa(y) ,y+yxb(y),z+yxc(y))$, so that
\[ \varphi_1^*\partial_{x} = (1+ya,yb,yc). \]
We see that for each vector field $X$ tangent to $\partial_{x}$ at $0$, its restriction to the $y$-axis can be obtained this way.
Note that $\varphi_1$ fixes $\{y=0\}$ and preserves the $y$-axis so that if we do the same construction for the $z$-axis, the newly constructed diffeomorphism $\varphi_2$ will preserve the vector field along the $y$-axis.
Hence $\varphi=\varphi_2 \varphi_1$ will conjugate $I(f,g)$ with $\langle x^2y,x^2z,yz \rangle$.
\end{proof}

\begin{thm}
Let $(f_0,g_0)$ and $(f_1,g_1)$ be two exceptional pairs of cusps on $\Cg{3}$ with tangency curves $T^i_j$ ($i=0,1$, $j=1,2,3$ and $T_1^i$ is the simple one).
Suppose that there is a diffeomorphism $\psi$ conjugating the tangency curves and the restrictions $(f_i\vert_{T^i_j},g_i\vert_{T^i_j})$.
Then there exists a diffeomorphism $\varphi$ such that $(f_0\circ \varphi,g_0\circ \varphi) = (f_1,g_1)$.
\end{thm}

\begin{proof}
After what has been done before, we can suppose that each couple is in the form \eqref{form_cusps}, with tangency ideals $I = \langle x^2y,x^2z,yz\rangle$, with $f_0=f_1$ everywhere and $g_0=g_1$ in restriction to the tangency locus $T$.

Let $X$ be a vector field characterizing the ideal $I$.
If $Y$ is tangent to $T$, then $\lambda X + \mu Y$ also characterizes $I$ for all $\lambda,\mu\in \OC{3}$ with $\lambda$ not vanishing on $T$ so we can suppose that $X \in \mr{Ker}(df_0)$ at every point of $T$ (note that $\mr{Ker}(df_0)$ is transverse to $T$ at each point different from the origin).
By definition of the tangency locus, $X$ then also belongs to the kernel of $dg_i$ for each $i$ on $T$, hence $g_1-g_0\in I$.

The key lemma can then be applied to finish the proof of this theorem.
\end{proof}

\end{document}